\documentclass[12pt]{amsart}
\usepackage[utf8]{inputenc}
\usepackage[T1]{fontenc}
\usepackage{amsmath,amsfonts,amsthm,amssymb}
\usepackage{enumitem}
\usepackage{graphicx}
\usepackage{pgf,tikz,pgfplots}
\usetikzlibrary{arrows}
\usetikzlibrary{angles,quotes}
\usetikzlibrary{calc}
\usepgfplotslibrary{fillbetween}
\usetikzlibrary{positioning}
\usetikzlibrary{calc,shapes.geometric}
\usetikzlibrary{decorations.pathreplacing}
\usetikzlibrary{cd}
\usetikzlibrary{positioning}
\usepackage{mathrsfs}
\usepackage{caption,subcaption}
\usepackage{mathtools}
\usepackage{xspace}
\usetikzlibrary{patterns}
\calclayout
\usepackage{soul}
\usepackage{csquotes}

\usepackage{hyperref}
\hypersetup{
  colorlinks   = true,          
  urlcolor     = blue,          
  linkcolor    = green,          
  citecolor   = orange             
}

\newtheorem{theorem}{Theorem}
\newtheorem{corollary}[theorem]{Corollary}
\newtheorem{conjecture}[theorem]{Conjecture}

\newtheorem{lemma}[theorem]{Lemma}

\theoremstyle{definition}
\newtheorem{definition}[theorem]{Definition}

\newtheorem{problem}[theorem]{Problem}

\theoremstyle{remark}
\newtheorem{remark}[theorem]{Remark}

\newcommand\doi[1]{\href{http://dx.doi.org/#1}{\texttt{doi:#1}}}

\graphicspath{{Figures/}}

\tikzstyle{lattice} = [draw=red, fill=red]
\tikzstyle{valid} = [draw=red, thick, fill=white]
\tikzstyle{intersect} = [draw=orange, fill=orange]
\tikzstyle{boundary} = [draw=blue, fill=blue]
\tikzstyle{triangle} = [draw=black, thick, fill=blue!20]
\tikzstyle{inequality} = [draw=green, thick]
\tikzstyle{someline} = [draw=black, dashed]

\title[Positroids, Dressian and Stable polynomials]{Positroids, Dressian and Stable polynomials}
\author{Ayush Kumar Tewari}

\address[A.~K.~Tewari]{
}
\email{tewari@math.tu-berlin.de}

\makeatletter
\@namedef{subjclassname@1991}{Mathematics Subject Classification (2020)}
\makeatother



\thanks{I would like to thank Peter Br\"andén and Michael Joswig for going through earlier drafts of this article and for providing valuable suggestions and comments. I would also like to thank Kevin Purbhoo for discussing his work and clearing any queries I had about them.}

\usepackage{pgfplots}
\pgfplotsset{compat=1.18}

\begin{document}

\begin{abstract}
Our work is motivated by the connection established between Lorentzian polynomials and the Dressian in the seminal work of Br\"and\'en and Huh on Lorentzian polynomials. We analyze this relation for the class of positroids, and are able to show that in this case, we can relate a multiaffine homogenous stable polynomial to it. Additionally, we also highlight that a conjecture for matroids posed by Br\"and\'en and Huh is true when considered over the class of Rayleigh matroids which strictly contain the class of positroids. We collect these findings along with other results for further exploration.  
\end{abstract}

\maketitle

{
\hypersetup{linkcolor=blue}
\tableofcontents
}

\section{Introduction}

The combinatorics associated with electrical networks carries a lot of really intricate details related to matroids, an initial documentation of which can be found in \cite{recski1989matroid}. In subsequent studies \cite{choe2004homogeneous} the relation between matroids and the geometry of polynomials is described. In recent work \cite{branden2020lorentzian} with the introduction of the notion of \emph{Lorentzian polynomials} it has been demonstrated that highly intrinsic properties concerning matroids can be obtained via findings that are based on the geometry of polynomials. For example, the well-known Mason conjecture for matroids can be resolved by methods developed in \cite{branden2020lorentzian, huh2021correlation}. These results are also extended to show negative dependence properties for the $q$-state Potts model partition function \cite{branden2018hodge}.

In \cite{choe2004homogeneous} the authors are interested in studying the \emph{half-plane property} for multivariate polynomials. For matroids, this is equivalent to the case when the basis generating polynomial is \emph{stable}.  The authors in \cite{choe2004homogeneous} also discuss the various examples of matroids that satisfy the half-plane matroids: uniform matroids, sixth-root of unity matroids, and all matroids of rank and corank two satisfy the half-plane property. Moreover, in recent work \cite{kummer2022matroids} the authors provide a complete classification of which matroids on at most $8$ elements have the half-plane property.  These investigations have also now forayed into asking slightly refined formulations, for example, the \emph{Rayleigh} property for polynomials which captures negative dependence correlation between elements of the bases in a matroid. In \cite{choe2006rayleigh} the authors introduce the notion of \emph{Rayleigh} matroid, which is the class of matroids whose basis generating polynomials is Rayleigh. They also provide examples of matroids which are Rayleigh: regular matroids and all matroids of rank three are Rayleigh.

In this paper, our emphasis is on the class of \emph{positroids}, which in essence captures positivity in the class of matroids. In one of the first results, we build on the work done in \cite{purbhoo2018total} in which they show that a certain multi affine homogenous polynomial over $\mathbb{C}$ which \emph{represents} a point in the Grassmannian is stable if and only if this point represents a nonegative point which corresponds to a positroid. We extend this result over real and real closed fields, and use this to refine the characterization of points in the Dressian of a positroid in the form of this 
result

\begin{theorem}
Let $V$ represent a nonegative point in $Gr(k,n)$ and let $\mathcal{M}$ be a positroid of rank $k$ on $[n]$ elements with bases $\mathcal{B}$ that corresponds to $V$. Let $p_{B}$ represent the nonzero \emph{Pl\"ucker coordinates} of $V$, which also satisfy the quadratic Plücker relations. Then for $\mu \in \text{Trop}\>Gr(\mathcal{M}) \subseteq  Dr(\mathcal{M})$, the multi affine homogenous stable polynomial $f$ that \emph{represents} $V$,
\[   f =    \sum_{B \in \mathcal{B}} p_{B} \> x^{B} \]
in $\mathbb{K}[x_{1} , \hdots , x_{n}]$ is the one whose tropicalization is $\mu$.
\end{theorem}

This result is sort of intermediate between the weak half-plane property and strong half-plane property; in the sense that in the case of the strong plane property, the polynomial being considered is fixed and it is the basis generating polynomial of the matroid, whereas, in the case of the weak half-plane property, it guarantees only the existence of a stable polynomial whose support is the bases of the matroid.

Additionally, we also share an observation to highlight that Conjecture 3.12 in \cite{branden2020lorentzian} is true when considered over the class of Rayleigh matroids \cite{choe2006rayleigh}. This follows by the definition used in \cite{choe2006rayleigh}, and as this class of matroids  strictly contains the class of positroids \cite{marcott2016positroids}, we refurbish this result in the following form, 

\begin{corollary}
The following conditions are equivalent for any non-empty $J \subseteq \{0, 1\}^{n}$ :

\begin{enumerate}
    \item $J$ is the set of bases of a Rayleigh matroid on $[n]$.
    \item The generating function $f_{J}$ is a homogeneous $1-$Rayleigh polynomial.
\end{enumerate}
Essentially, Conjecture 3.12 \cite{branden2020lorentzian} is true for the class of Rayleigh matroids.    
\end{corollary}

We conclude with an extension of results relating Lorentzian polynomials and Dressians to the case of Flag Dressians and we also list some problems which we aim to work on in the future.

\section{Preliminaries}

We try to introduce the basic notions in this section which we use in our subsequent sections. We refer the reader to \cite{oxley} for details on matroid theory.

A \emph{matroid} $\mathcal{M}$ of rank $k$ on the set $E$ is a nonempty collection $\mathcal{B} \subseteq \binom{E}{k}$ of $k$-element subsets of $E$, called \emph{bases} of $\mathcal{M}$, that satisfies the basis exchange axiom: 
For any $I , J \in \mathcal{B}$ and $a \in I$, there exists $b \in J$ such that $I \setminus \{ a \} \cup \{ b \} \in \mathcal{B}$.

A matroid is called \emph{representable} if it can be represented by columns of a matrix over some field $\mathbb{K}$. We index the columns of a $k \times n$ matrix by the set $[n]$.

The generating function of a matroid $\mathcal{M}$ with bases $\mathcal{B}$ (also referred to as the basis enumeration polynomials in \cite{marcott2016positroids}) is defined as the polynomial

\begin{equation}\label{eq:gen_func_matroid}
     M(x)  = \sum_{B \in \mathcal{B}} x^{B}   
\end{equation}

A slightly general notion of \emph{generating function} of a subset $J \subseteq \mathbb{N}^{n}$, is defined as follows

\[ f_{J}  = \sum_{\alpha \in J} \frac{x^{\alpha}}{\alpha!} \quad \text{where} \quad  \alpha! = \prod_{i=1}^{n} \alpha_{i}! \] 

and we see that these two notions agree when $J \subseteq \{0, 1\}^{n}$ .

The \emph{Grassmannian} $Gr(k,n)$ is the parameterization of the family of all $k$-dimensional subspaces of $n$-dimensional vector space in $\mathbb{K}^{n}$. It enjoys a smooth projective variety structure, corresponding to the vanishing set of the \emph{Pl\"ucker ideal}. 

An element $V$ in the Grassmannian $Gr(k,n)$ is understood as a collection of $n$ vectors $v_{1}, \hdots, v_{n} \in \mathbb{K}^{k}$ spanning the space $\mathbb{K}^{k}$ modulo the simultaneous action of $GL(k,n)$. Let $A$ be a $k \times n$-matrix consisting of column vectors $v_1, v_2, \ldots, v_n$. We call $A$, a \emph{representative} of $V$. This defines a matroid $\mathcal{M}_V$ whose bases are the $k$-subsets $I \subseteq [n]$ such that $\text{det}(A[I]) \neq 0$. It is important to note that this is independent of the choice of $A$, and only depends on $V$. Here, $\text{det}(A[I])$ denotes the determinant of the $k \times k$ submatrix of $A$ with the column set $I \in \binom{[n]}{k}$.

A \emph{positroid} $P$ of rank $k$ is a matroid that can be represented by a $k \times n$-matrix $A$ such that the maximal minor $p_I$ is non-negative for each $I \in \binom{[n]}{k}$.

We now introduce the definition of \emph{Lorentzian} polynomials \cite{branden2020lorentzian}. We mostly rely on \cite{branden2020lorentzian} and \cite{dcamurota} for our definitions. Let $n$ and $d$ be nonnegative integers, and set $[n] = \{1, \hdots , n\}$ and  $H^{d}_{n}$ to be the set of degree $d$ homogeneous polynomials in $\mathbb{R}[w_{1} , \hdots , w_{n}]$. Let $f$ be a polynomial in $\mathbb{R}[w_{1} , \cdots , w_{n}]$,

\[ f(x) = \sum_{\alpha \in \mathbb{N}^{n}} \frac{c_{\alpha}}{\alpha!} x^{\alpha}   \]

The \emph{support} of the polynomial $f$ is the subset of $\mathbb{N}^{n}$ n defined by

\[  \text{supp}(f) = \{\alpha \in \mathbb{N}^{n} \>\> | \>\> c_{\alpha} \neq 0 \} \]

Let $P^{d}_{n} \subset  H^{d}_{n}$ be the open
subset of polynomials all of whose coefficients are positive. The \emph{Hessian} of $f \in \mathbb{R}[w_{1}, \cdots , w_{n}]$ is the symmetric matrix

\[ \mathcal{H}_{f}(w)  = (\partial_{i} \partial_{j}f )^{n}_{i,j=1} \]

where $\partial$ denotes the partial derivative $\frac{\partial}{\partial w_{i}}$. 

\begin{definition}
Set $L^{0}_{n}  = P^{0}_{n}, L^{1}_{n}  = P^{1}_{n}$ , and 

\[ L^{2}_{n} = \{ f \in P^{2}_{n} \> \> | \> \> \mathcal{H}_{f} \> \> \text{is nonsingular and has exactly one positive eigenvalue} \}  \]

For $d\geq 2$, $L^{d}_{n}$ is defined recursively as follows 

\[ L^{d}_{n} = \{ f \in P^{d}_{n} \> \> | \> \> \partial_{i}f \in L^{d-1}_{n} \> \> \text{for all}  \> \> i \in [n] \}  \]

The polynomials in $L^{d}_{n}$ are called \emph{strictly Lorentzian}, and the limits of strictly Lorentzian polynomials are called \emph{Lorentzian}. 
\end{definition}

$L^{d}_{n}$ can also be identified with the set of $n \times n$ symmetric matrices with positive entries that have the \emph{Lorentzian signature} $(+, -, . . . , -)$. 

\begin{definition}
A subset $J \subset \mathbb{N}^{n}$  to be $M$-convex if it satisfies any one of the following equivalent conditions:  

\begin{itemize}
    \item For any $\alpha, \beta \in J$ and any index $i$ satisfying $\alpha_{i} > \beta_{i}$, there is an index $j$ satisfying
    \[ \alpha_{j} < \beta_{j} \quad \text{and} \quad \alpha - e_{i} + e_{j} \in J   \]
    \item For any $\alpha, \beta \in J$ and any index $i$ satisfying $\alpha_{i} > \beta_{i}$, there is an index $j$ satisfying
    \[ \alpha_{j} < \beta_{j} \quad \text{and} \quad \alpha - e_{i} + e_{j} \in J \quad \text{and} \quad \beta - e_{i} + e_{j} \in J  \]

\end{itemize}
\end{definition}

Let $\mu$ be a function from $\mathbb{N}^{n}$ to $\mathbb{R} \cup \{\infty\}$. The effective domain of $\mu$ is, by definition,

\[  \text{dom}(\mu)  = \{ \alpha \in \mathbb{N}^{n} \>\> | \>\>  \mu(\alpha) < \infty  \}  \]

\begin{definition}
A function $\mu$ from $\mathbb{N}^{n}$ to $\mathbb{R} \cup \{\infty\}$ is said to be $M$-convex if it satisfies the \emph{symmetric exchange property}:
\begin{itemize}
    \item For any $\alpha, \beta \in \text{dom}(\mu)$ and any $i$ satisfying $\alpha_{i}  > \beta_{i}$ , there is $j$ satisfying
    \[  \alpha_{j} < \beta_{j} \quad \text{and} \quad \mu(\alpha) + \mu(\beta) \geq \mu(\alpha - e_{i} + e_{j} ) + \mu(\beta - e_{j} + e_{i} ).     \]    
\end{itemize}
\end{definition}

The effective domain of an $M$-convex function on $\mathbb{N}^{n}$ is an $M$-convex subset of $\mathbb{N}^{n}$. When the effective domain of a function $\mu$ is is $M$-convex, the symmetric exchange property for $\mu$ is equivalent to the following \emph{local exchange property},

\begin{itemize}
    \item For any $\alpha, \beta \in \text{dom}(\mu)$ with $|\alpha - \beta|_{1} = 4$, there exist $i$ and $j$ satisfying
    \[  \alpha_{i} > \beta_{i} , \alpha_{j} < \beta_{j} \quad \text{and} \quad \mu(\alpha) + \mu(\beta) \geq  \mu(\alpha - e_{i} + e_{j} ) + \mu(\beta - e_{j} + e_{i} ).   \]
\end{itemize}

A function $\mu : \mathbb{N}^{n} \rightarrow \mathbb{R} \cup \{\infty\}$ is said to be $M$-concave if $-\mu$ is $M$-convex. The effective domain of an M-concave function $\mu$ is,

\[  \text{dom}(\mu)  = \{ \alpha \in \mathbb{N}^{n} \>\> | \>\>  \mu(\alpha) > -\infty  \}  \]

A \emph{valuated matroid} on $[n]$ is an $M$-concave function on $\mathbb{N}^{n}$ whose effective
domain is a nonempty subset of $\{0, 1\}^{n}$ . The effective domain of a valuated matroid $\mu$ on $[n]$ is the set of bases of a matroid on $[n]$, the \emph{underlying matroid} of $\mu$.

Let $M^{d}_{n}(\mathbb{K})$ denote the set of all all degree $d$ homogeneous polynomials in $K_{\geq 0}[w_{1} , \cdots, w_{n}]$ whose support is $M$-convex.

Although, the standard definition of Lorentzian polynomials is over $\mathbb{R}$ in \cite{branden2020lorentzian} itself the authors provide a definition over any field $\mathbb{K}$ as follows,

\begin{definition}
Set $L^{0}_{n}(\mathbb{K}) = M^{0}_{n}(\mathbb{K}), L^{1}_{n}(\mathbb{K}) = M^{1}_{n}(\mathbb{K})$, and    
\[ L^{2}_{n}(\mathbb{K})  = \{ f_{t} \in M^{2}_{n}(\mathbb{K}) \>\> | \>\> \text{The Hessian of} \>\> f_{t} \>\> \text{has at most one eigenvalue in} \>\> \mathbb{K}_{\geq 0} \} \]

For $d \geq 3$, we define $L^{d}_{n}(\mathbb{K})$ as follows
\[   L^{d}_{n}(\mathbb{K})  = \{ f_{t} \in M^{d}_{n}(\mathbb{K}) \>\> | \>\> \partial^{\alpha}f_{t} \in L^{2}_{n}(\mathbb{K})  \>\> \text{for all} \>\> \alpha \in \Delta^{d-2}_{n} \}  \]

where $\Delta^{d-2}_{n}$ represents the discrete $(d-2)-$th simplex.
\end{definition}

We now discuss the background of the definition of stable polynomials and how they are related to the study of Lorentzian polynomials. In earlier works, the definition of stable polynomials is referred to as the \emph{half plane property} \cite{choe2004homogeneous}. Let $\mathcal{H} = \{ x \in \mathbb{C} \>\> | \>\> \text{Re} x > 0 \}$  denote the open upper half plane.

\begin{definition}
A polynomial $f$ in $\mathbb{R}[w_{1} , \cdots , w_{n}]$ is called \emph{stable} if $f$ is nonvanishing on $\mathcal{H}^{n}$ or identically
zero.   

In this case, $f$ is also said to satisfy the \emph{half plane property}.
\end{definition}

\begin{remark}
In \cite{choe2004homogeneous}, the half-plane property is also defined for a matroid. $\mathcal{M}$ is said to satisfy the \emph{half plane property} if the basis generating polynomial of $\mathcal{M}$ satisfies the half-plane property. Additionally, $\mathcal{M}$ is said to satisfy the \emph{weak half plane} property if there exists a polynomial that satisfies the half-plane property and whose support is the basis of $\mathcal{M}$ \cite{choe2004homogeneous}.
\end{remark}

Let $S^{d}_{n}$ be the set of degree $d$ homogeneous stable polynomials in $n$ variables with nonnegative coefficients and by Hurwitz Theorem $S^{d}_{n} \subset H^{d}_{n}$. We refer the reader to \cite{wagner2011multivariate} to explore the rich theory of stable polynomials. Any polynomial $f \in S^{d}_{n}$  is the limit of \emph{strictly stable polynomials} \cite{nuij1968note}.

Lorentzian polynomials are a generalization of stable polynomials, i.e., all stable polynomials are Lorentzian, moreover, any homogeneous stable polynomial is a constant multiple of a Lorentzian polynomial \cite{branden2020lorentzian}.

In \cite{choe2004homogeneous} the authors envisaged that for large classes of matroids, the \emph{basis generating polynomials} are stable. Subsequent studies also concentrate on certain relaxations of this class, the most prominent of them being \emph{Rayleigh} polynomials,

\begin{definition}
Let $c$ be a fixed positive real number, and let $f$ be a polynomial in $\mathbb{R}[w_{1} , \cdots , w_{n}]$. $f$ is  called $c$-\emph{Rayleigh} if $f$ has nonnegative coefficients and  

\[ \partial^{\alpha}f(x) \> \partial^{\alpha+e_{i}+e_{j}}f(x) \leq c \> \partial^{\alpha+e_{i}}f(x) \> \partial^{\alpha+e_{j}}f(x) \quad \text{for all} \>\> i,j \in [n], \alpha \in \mathbb{N}^{n}, x \in \mathbb{R}^{n}_{\geq 0}   \]

When $f$ is multi-affine, that is, when $f$ has degree at most one in
each variable, the $c$-Rayleigh condition for $f$ is equivalent to

\[ f(x) \partial^{i} \partial^{j}f(x) \leq c \>\> \partial^{i}f(x) \partial^{j}f(x) \>\> \text{for all distinct} \>\> i,j \in \mathbb{N}^{n}, \>\> \text{and} \>\> x \in \mathbb{R}^{n}_{\geq 0}  \]
\end{definition}

A multi-affine polynomial $f$ is said to be \emph{strongly Rayleigh} if

\[ f(x) \partial^{i} \partial^{j}f(x) \leq c \>\> \partial^{i}f(x) \partial^{j}f(x) \>\> \text{for all distinct} \>\> i,j \in \mathbb{N}^{n}, \>\> \text{and} \>\> x \in \mathbb{R}^{n}  \]

A multi-affine polynomial is stable if and only if it is strongly Rayleigh \cite[Theorem 5.6]{branden2007polynomials}. Based on the definition of Rayleigh polynomials, the notion of \emph{Rayleigh} matroids \cite{choe2006rayleigh} was introduced, where a matroid $\mathcal{M}$ is a Rayleigh matroid if its basis generating polynomial is Rayleigh. Similarly, there also exists the class of \emph{strongly Rayleigh matroid}, where a matroid is a \emph{strongly Rayleigh matroid} if its basis generating polynomial is strongly Rayleigh. It is straightforward to see from \cite[Theorem 5.6]{branden2007polynomials} that all strongly Rayleigh matroids satisfy the half-plane property.

We now define notions related to tropical geometry that appear in our discussion and we refer the reader to \cite{maclagan2021introduction} for further details. Tropical geometry is the study of vanishing sets of polynomials defined over the \emph{tropical} semiring $\mathbb{T} = \{\mathbb{R} \cup \{-\infty\}, \text{max} = \oplus, + = \odot\}$. A \emph{tropical} polynomial is a polynomial defined over $\mathbb{T}$ with the binary operations replaced by $\oplus$ and $\odot$. $\text{Trop}(f)$ denotes the \emph{tropical hypersurface} associated with $f$ which is the collection of all points where the maxima is achieved at least twice.

The \emph{tropical Grassmannian} $\text{TropGr}(k,n)$ is the intersection of the tropical hypersurfaces $\text{Trop}(f)$, where $f$ ranges over all elements of the \emph{Pl\"ucker ideal} $\mathcal{I}_{k,n}$ which is generated by the \emph{quadratic Pl\"ucker relations}, and therefore it is also a tropical variety \cite{maclagan2021introduction}. The \emph{Dressian} $Dr(k,n)$ is the intersection of the tropical hypersurfaces $\text{Trop}(f)$, where $f$ ranges over all three-term Pl\"ucker relations, that generate the \emph{Pl\"ucker ideal} and hence it possesses the structure of a tropical prevariety \cite{maclagan2021introduction}. The underlying matroid for the definitions of the tropical Grassmannian and Dressian is the \emph{uniform matroid} $\mathcal{U}_{k,n}$. However, the notion of Dressian has been extended to arbitrary matroids with the idea of the \emph{local Dressian}. The \emph{local Dressian} $Dr(M)$ is defined as the tropical pre-variety given by the set of quadrics obtained from the three-term  Pl\"ucker relations by setting the variables $p_{B}$ to zero, where $B$ is not a basis of $M$ \cite{olarte2019local}.

\begin{remark}
We acknowledge that in some instances, all Pl\"ucker relations (not necessarily only the 3-term Pl\"ucker relations) are considered for the definition of the Dressian. However, by the work in \cite{baker2019matroids} we know that these two notions coincide over the tropical hyperfield and hence our definition would remain consistent for our discussion.    
\end{remark}

\section{Dressian, Stable and Lorentzian Polynomials}

Since the introduction of the notion of a Dressian, it has been noted that the points residing in the Dressian enjoy multiple avatars in which they can be treated and satisfy multiple properties. The subject for initial findings was the Dressian $Dr(k,n)$ with the underlying matroid being the uniform matroid $\mathcal{U}_{k,n}$. We collect the various previously known \cite{maclagan2021introduction,speyer2008tropical,branden2020lorentzian,tewari2022generalized} equivalent notions for a point residing in the Dressian $Dr(k,n)$ in the form of Theorem \ref{thm:Dressian}. 

\begin{theorem}\label{thm:Dressian}
Let $\mu$ be a point in $Dr(k,n)$. Then the following are equivalent

\begin{enumerate}
    \item $\mu$ is a valuated matroid with the underlying matroid $\mathcal{U}_{k,n}$.
    \item $\mu$ defines a $M-convex$ function on the matroid $\mathcal{U}_{k,n}$.
    \item $\mu$ satisfies the tropical three-term Pl\"ucker relations.
    \item $\mu$ as a weight vector induces a regular matroidal subdivision of $\Delta(k,n)$.
    \item $\mu = \text{trop}(f_{t})$, where $f_{t}$ is a Lorentzial polynomial defined on the bases set $\mathcal{B}$ of $\mathcal{U}_{k,n}$.
    \item $\mu$ defines a metric for a generalised metric tree arrangement.
\end{enumerate}
\end{theorem}

These equivalences can also be extended to the local Dressian $Dr(\mathcal{M})$ \cite{olarte2019local}, which can be defined for an arbitrary matroid $\mathcal{M}$, and it generalizes the notion of Dressian $Dr(k,n)$, and we list the corresponding equivalences for a point residing in the local Dressian in Theorem \ref{thm:Dressian_local},

\begin{theorem}\label{thm:Dressian_local}
Let $\mu$ be a point in $Dr(\mathcal{M})$. Then the following are equivalent

\begin{enumerate}
    \item $\mu$ is a valuated matroid with the underlying matroid $\mathcal{M}$.
    \item $\mu$ defines a $M-convex$ function on the matroid $\mathcal{M}$.
    \item $\mu$ satisfies the tropical three-term Pl\"ucker relations.
    \item $\mu$ as a weight vector induces a regular matroidal subdivision of $\mathcal{P}_{M}$.
    \item $\mu = \text{trop}(f_{t})$, where $f_{t}$ is a Lorentzial polynomial defined on the bases set $\mathcal{B}$ of $\mathcal{M}$.
\end{enumerate}
\end{theorem}

The categorization given in (v) is one of the most recent ones, proven in 
\cite[Theorem 3.20]{branden2020lorentzian}. 

Our aim now is to possibly refine the characterization of the Dressian, concerning Lorentzian polynomials. We know that the class of Loentzian polynomials contains the class of stable polynomials \cite{branden2020lorentzian}. Of interest in this regard is the work in \cite{purbhoo2018total}, in which the author relates the notions of stability and total nonnegativity. We first briefly describe the setup used in \cite{purbhoo2018total} to make ideas clearer.  Consider the matrix $M \in \text{Mat(k,n)}$ of rank $k$ over the field of complex number. The column space of the matrix provides us a k-dimensional subspace $V$ which lies in the $Gr(k,n)$. Let $M[I]$ denote the $k \times k$ submatrix of $M$ with row set $I \in \binom{n}{k}$. The \emph{Pl\"ucker coordinates} of $V$ are the maximal minors $ 
\Bigl\{ \text{det}(M [I]) : I \in \binom{n}{k} \Bigr\}$. The homogenous multi-affine polynomial, 

\begin{equation}\label{eq:ply_nonnegative}
 f =    \sum_{I \in \binom{n}{k}} \text{det}(M [I]) \> x^{I}
\end{equation}    

where $x^{I} = \prod_{i \in I} x_{i}$, is said to \emph{represent} $V$ in $Gr(k,n)$ \cite{purbhoo2018total}. Additionally, a necessary and sufficient condition for a polynomial to represent $V$ is that the coefficients satisfy the quadratic \emph{Pl\"ucker} relations, which are the defining equations for the $Gr(k,n)$, and provide it the structure of a projective variety. We recall the following result from \cite{purbhoo2018total} concerning the multi affine polynomial in Equation \ref{eq:ply_nonnegative}

\begin{theorem}[Theorem 1.1 \cite{purbhoo2018total}]\label{thm:complex_stable}
Let $f \in \mathbb{C}[x]$ be a multi-affine homogenous polynomial of degree $k$ that represents a point $V$ in $Gr(k,n)$. Then, $f$ is stable if and only if $V$ is totally nonnegative.   
\end{theorem}

We state the above result over the field of real numbers,

\begin{corollary}\label{cor:real_stable}
Let $f \in \mathbb{R}[x]$ be a multi-affine homogenous polynomial of degree $k$ that represents a point $V$ in $Gr(k,n)$. Then, $f$ is stable if and only if $V$ is totally nonnegative. Moreover, if $\mathbb{K}$ is a real closed field and $f \in \mathbb{K}[x]$ be a multi-affine homogenous polynomial of degree $k$ that represents a point $V$ in $Gr(k,n)$ then, $f$ is stable if and only if $V$ is totally nonnegative.   
\end{corollary}

\begin{proof}
With the well-known \enquote{phase theorem} \cite[Theorem 6.1]{choe2004homogeneous} we know that every homogeneous stable polynomial has non-negative real coefficients up to scalar multiples. Hence, the stability of the polynomial in Theorem \ref{thm:complex_stable} over $\mathbb{C}$ is equivalent to stability of over $\mathbb{R}$ \cite{purbhooemail}.  

Since, the field of real numbers also allows quantifier elimination, therefore by invoking the Tarski principle as done in \cite{branden2020lorentzian}
, we can extend Corollary \ref{cor:real_stable} over real closed fields as well.
\end{proof}

We also know that $V \in Gr(k,n)$ determines a representable matroid of rank $k$ on the set $[n]$, by taking the bases to be the indices of the nonzero Pl\"ucker coordinates. If $V$ is totally nonnegative, this matroid is a positroid \cite{purbhoo2018total}. With this observation, we want to consider a restriction of a result on matroids \cite[Theorem 3.20]{branden2020lorentzian} to the class of positroids. We first recall this result,

\begin{theorem}[Theorem 3.20 \cite{branden2020lorentzian}]\label{thm:lorent}
The following conditions are equivalent for any function $\mu : \Delta(k,n) \rightarrow \mathbb{Q} \cup \{ \infty \} $

\begin{enumerate}
    \item the function $\mu$ is M-convex;
    \item there is a Lorentzian polynomial in $\mathbb{K}[w_{1} , \hdots , w_{n}]$ whose tropicalization is $\mu$.
\end{enumerate}
\end{theorem}

We also elaborate on the background setup for this statement used in \cite{branden2020lorentzian}. In this statement, the field $\mathbb{K}$ is the real convergent Puiseux series 

\[ \mathbb{K} = \bigcup_{k \geq 1} \mathbb{R}((t^{1/k}))_{\text{conv}} \]

which is a \emph{real closed field} and its corresponding algebraic closure is the field 

\[  \overline{\mathbb{K}} = \bigcup_{k \geq 1} \mathbb{C}((t^{1/k}))_{\text{conv}}  \]

The definition of Lorentzian polynomials over $\mathbb{K}$ are also provided in \cite[Definition 3.17]{branden2020lorentzian}. We propose the following refinement of Theorem \ref{thm:lorent}, when we restrict ourselves to the class of positroids, in which case we see that we obtain a relation to homogenous stable polynomials which are a subclass of Lorentzian polynomials. 

\begin{theorem}\label{thm:positroid_stable}
Let $V$ represent a nonnegative point in $Gr(k,n)$ and let $\mathcal{M}$ be a positroid of rank $k$ on $[n]$ elements with bases $\mathcal{B}$ that corresponds to $V$. Let $p_{B}$ represent the nonzero \emph{Pl\"ucker coordinates} of $V$, which also satisfy the quadratic Plücker relations. Then for $\mu \in \text{Trop}\>Gr(\mathcal{M}) \subseteq  Dr(\mathcal{M})$, the multi affine homogenous stable polynomial $f$ that \emph{represents} $V$,
\[   f =    \sum_{B \in \mathcal{B}} p_{B} \> x^{B} \]
in $\mathbb{K}[x_{1} , \hdots , x_{n}]$ is the one whose tropicalization is $\mu$.
\end{theorem}

\begin{proof}

 We consider the homogenous multi-affine polynomial, 

\begin{equation}
 f =    \sum_{B \in \mathcal{B}} p_{B} \> x^{B}
\end{equation}    

where $x^{B} = \prod_{b \in B} x_{b}$, which \emph{represents} $V$. By Corollary \ref{cor:real_stable}, this polynomial is stable over $\mathbb{K}$ with coefficients satisfying the quadratic Pl\"ucker relations. We also realize that the tropicalization $\text{trop}(f)$ is defined by the tropicalization of the coefficients which are just Pl\"ucker coordinates. However, these tropical Pl\"ucker relations are the relations that define a point in the tropical Grassmannian $\text{TropGr}(k,n)$. Therefore, we realize that $\text{trop}(f)$ has coefficients that satisfy the tropical Pl\"ucker relations, which defines a point $\mu \in \text{TropGr}(k,n) \subseteq  Dr(\mathcal{M})$. 
\end{proof}

We also note that this relation between stable polynomials and positroids does not extend to the class of matroids, for example in \cite{branden2007polynomials} it is shown that if we consider the matroid $\mathcal{M}$ as the Fano plane then there is no stable polynomial whose
support is $B$. However, the existence of a homogenous multi affine stable polynomial over the support of a matroid is known for matroids representable over $\mathbb{C}$ \cite[Corollary 8.2]{choe2004homogeneous}. We want to highlight that Theorem \ref{thm:positroid_stable} emphasizes the fact that for the class of positroids, we can actually provide the exact form of the homogenous multi affine stable polynomial whose support is the basis of the positroid and whose tropicalization resides as a point in the Dressian.  

Another result from \cite{branden2020lorentzian} that we are interested in studying is the conjecture concerning the generating function of matroids, which the authors state as follows

\begin{conjecture}[Conjecture 3.12 \cite{branden2020lorentzian}]\label{conj:lorent_conj}
The following conditions are equivalent for any non-empty $J \subseteq \{0, 1\}^{n}$ :

\begin{enumerate}
    \item $J$ is the set of bases of a matroid on $[n]$.
    \item The generating function $f_{J}$ is a homogeneous $\frac{8}{7}$ -Rayleigh polynomial.
\end{enumerate}
\end{conjecture}

The constant $\frac{8}{7}$ appearing in this conjecture is strict in the sense that for any positive real number $c < \frac{8}{7}$, there is a matroid whose basis generating function is not $c$-Rayleigh \cite{huh2021correlation}.

We now highlight that this conjecture is true for the class of Rayleigh matroids which strictly contain the class of positroids \cite{marcott2016positroids} and moreover, a stronger statement is true for them in the form of Theorem \ref{thm:positroid_rayleigh}.

\begin{corollary}\label{thm:positroid_rayleigh}
The following conditions are equivalent for any non-empty $J \subseteq \{0, 1\}^{n}$ :

\begin{enumerate}
    \item $J$ is the set of bases of a Rayleigh matroid on $[n]$.
    \item The generating function $f_{J}$ is a homogeneous $1-$Rayleigh polynomial.
\end{enumerate}
Essentially, Conjecture \ref{conj:lorent_conj} is true for the class of Rayleigh matroids.    
\end{corollary}

\begin{proof}
For the implication (i) $\implies$ (ii) by the definition of Rayleigh matroids in \cite{choe2006rayleigh}, the basis generating function of a Rayleigh matroid is 1-Rayleigh. Also, if $f$ is a 1-Rayleigh polynomial, then $f$ is $c$-Rayleigh polynomial for all $c \geq 1$. Additionally, for the opposite implication i.e., (ii) $\implies$ (i) is straightforward in this case by \cite[Theorem 3.10]{branden2020lorentzian}, which implies that $J$ is the basis of a matroid, which in turn is Rayleigh by our assumption in this implication. Therefore, this helps us to see that Conjecture \ref{conj:lorent_conj} is true for the class of Rayleigh matroids.    
\end{proof}

An interesting observation is that the statement of Corollary \ref{thm:positroid_rayleigh} when considered over the class of positroids, fails to hold true which is confirmed from the example of the \emph{Vamos} matroid, which is not representable over any field and hence is not a positroid, even though it is 1-Rayleigh. This also explains the strict inclusion of positroids in the class of Rayleigh matroids.

\section{Future Work}

Motivated by Theorem \ref{thm:Dressian}, we now consider the case of the positive Dressian $Dr^{+}(k,n)$, which is based on the definition of the positive part of a tropical variety introduced in \cite{speyer2005tropical}. A very significant result pertaining to the positive Dressian is the following \cite{speyer2021positive},

\begin{theorem}[Theorem 3.9 \cite{speyer2021positive}]
The positive tropical Grassmannian  $\text{Trop}^{+} Gr(k,n)$ equals the positive Dressian $Dr^{+}(k,n)$. 
\end{theorem}

A version of this theorem in terms of the local Dressian is proven in \cite{arkani2021positive}

\begin{theorem}[Theorem 9.2 \cite{arkani2021positive}]\label{thm:posDres_equal_postropGrass}
Let $\mathcal{M}$ be a positroid. Then the positive Dressian   $Dr^{+}(\mathcal{M})$ equals the positive tropical positroidal cell $\text{Trop}^{+} \Pi_{\mathcal{M}}$. 
\end{theorem}

Theorem \ref{thm:Dress_pos} collects all previously known equivalent notions \cite{speyer2021positive, m=2amplut, arkani2021positive} of a point residing in the positive Dressian and captures the essence of Theorem \ref{thm:Dressian},

\begin{theorem}
Let $\mu$ be a point in $Dr^{+}(d,n)$. Then the following are equivalent

\begin{enumerate}
    \item $\mu$ is a realizable valuated matroid with the underlying positroid $\mathcal{U}_{d,n}$.
    \item $\mu$ defines a realizable $M-convex$ function.
    \item $\mu$ satisfies the positive tropical three-term Pl\"ucker relations.
    \item $\mu$ as a weight vector induces a regular positroidal subdivision of $\Delta(d,n)$.
    \item $\mu = \text{trop}(f_{t})$, where $f_{t}$ is a Lorentzial polynomial defined on the bases set $\mathcal{B}$ of the positroid $\mathcal{U}_{d,n}$.
\end{enumerate}
\end{theorem}    

We now present an extension of this theorem to the case of an arbitrary positroid in Theorem \ref{thm:pos_Dress_equiv}, which mostly builds on previously known results in \cite{arkani2021positive, speyer2021positive}.

\begin{theorem}\label{thm:pos_Dress_equiv}
Let $\mathcal{M}$ be a positroid and $\mu$ be a point in $Dr^{+}(\mathcal{M})$. Then the following are equivalent

\begin{enumerate}\label{thm:Dress_pos}
    \item $\mu$ is a realizable valuated matroid with the underlying positroid $\mathcal{M}$.
    \item $\mu$ defines a $M-convex$ function on the positroid $\mathcal{M}$.
    \item $\mu$ satisfies the positive tropical three-term Pl\"ucker relations.
    \item $\mu$ as a weight vector induces a regular positroidal subdivision of $\mathcal{P}_{\mathcal{M}}$.
    \item $\mu = \text{trop}(f_{t})$, where $f_{t}$ is a Lorentzial polynomial defined on the bases set $\mathcal{B}$ of the positroid $\mathcal{M}$.
\end{enumerate}
\end{theorem}

\begin{proof}
The equivalence between (i) $\iff$ (ii) is by definition. The equivalences for (ii) $\iff$ (iii) are established in \cite[Proposition 8.2]{arkani2021positive}. We concentrate on proving the equivalence of (i). We first concentrate on proving the equivalence (i) $\iff$ (iii). We already know that a point $\mu \in Dr^{+}(\mathcal{M}) \implies \mu \in Dr(\mathcal{M})$, hence the implication that $\mu$ defines a valuated matroid follows from the inclusion in the Dressian $Dr(\mathcal{M})$, the only nontrivial part being the realizability of $\mu$. The stratification of the tropical positive Grasmmanian for a positroid is discussed in \cite{arkani2021positive} and it is defined in terms of positroidal cells as 

\begin{equation}\label{eq:dress_real}
\text{Trop}^{+} \Pi_{\mathcal{M}}  = \text{val}(\Tilde{\Pi}_{\mathcal{M}}(\mathcal{R}_{\geq 0})) = Dr^{+}(\mathcal{M})    
\end{equation}

where $\mathcal{R}_{\geq 0}$ represents the positive part of the generalized Puiseux series and $val()$ represents the \emph{valuaton} map. The equality in \ref{eq:dress_real} implies that all positive Pl\"ucker vectors are realizable therefore by \cite[Theorem 9.2]{arkani2021positive}, \cite[Remark 3.1.7]{brandt2021tropical} implies that the valuated matroid $\mu$ on the positroid $\mathcal{M}$ is realizable. So (i) $\iff$ (iii). The implication in (v) is straightforward with the inclusion of $\mu \in Dr(\mathcal{M})$ restricted to a positroid. 

\end{proof}

We also look at a counterpart of the Dressian, defined in the case of flag matroids, the \emph{Flag Dressian} introduced in \cite{brandt2021tropical} $FlDr(M)$, where $M := (M_{1}, \hdots M_{n})$ is a flag matroid, and it is the tropical variety cut out by the incidence-Pl\"ucker relations and the Grassmann-Pl\"ucker relations. A classification result similar to Theorem \ref{thm:Dressian} is proven in \cite{brandt2021tropical}

\begin{theorem}[Theorem A \cite{brandt2021tropical}]\label{thm:FlagDressian_equ}
Let $\mu = (\mu_{1} , \hdots, \mu_{k})$ be a sequence of valuated matroids such that its sequence of underlying matroids $M = (M_{1}, \hdots, M_{k})$ is a ﬂag matroid. Then the following are equivalent:

\begin{enumerate}
    \item $\mu$ is a point on $FlDr(M)$, i.e. it satisﬁes tropical incidence-Pl\"ucker relations and Grassmann-Pl\"ucker relations,
    \item $\mu$ is a valuated ﬂag matroid with underlying ﬂag matroid $M$.
    \item $\mu$ induces a subdivision of the base polytope of $M$ into base polytopes of ﬂag matroids,
    \item the projective tropical linear spaces $\overline{trop(\mu_{i})}$ form a ﬂag $\overline{trop(\mu_{1})} \subseteq \hdots \subseteq \overline{trop(\mu_{k})}$
\end{enumerate}
\end{theorem}

Our first move concerning the Flag Dressian is establishing the equivalence between points in the Flag Dressian and a sequence of Lorentzian polynomials.

\begin{theorem}
Let $\mu = (\mu_{1}, \hdots, \mu_{k})$ be a sequence of valuated matroids such that its sequence of underlying matroids $M = (M_{1}, \hdots, M_{k})$ is a ﬂag matroid. Then the following are equivalent 

\begin{enumerate}
    \item $\mu$ is a point on $FlDr(M)$,
    \item there exist a sequence of Lorentzian polynomials $f_{1}, \hdots f_{k}$ such that $\mu$ = $(\text{trop}(f_{1}) = \mu_{1}, \hdots, \text{trop}(f_{k})= \mu_{k})$, where $(\mu_{1}, \hdots \mu_{k})$ forms a valuated flag matroid and $f_{i}$ is a Lorentzial polynomial defined on the bases set $B_{i}$ of $M_{i}$.
\end{enumerate}
\end{theorem}

\begin{proof}
We first consider the implication (i) $\implies$ (ii). Each valuated matroid $\mu_{i}$ corresponds to a M-convex function on $M_{i}$ by definition, and by \cite[Theorem 3.20]{branden2020lorentzian} this further corresponds to a Lorentzian polynomial $f_{i}$ such that $\text{trop}(f_{i}) = \mu_{i}$. Therefore, for $\mu \in FlDr(M)$, there exists a sequence of Lorentzian polynomials $f_{1}, \hdots f_{k}$, such that $\mu$ = $(\text{trop}(f_{1}) = \mu_{1}, \hdots, \text{trop}(f_{k}) = \mu_{k})$, where $f_{i}$ is a Lorentzial polynomial defined on the bases set $B_{i}$ of $M_{i}$. We now try to prove the implication (ii) $\implies$ (i). Given a sequence of Lorentzian polynomials such that $\mu$ = $(\text{trop}(f_{1}) = \mu_{1}, \hdots, \text{trop}(f_{k})= \mu_{k})$ is a valuated flag matroid on the underlying matroid $M = (M_{1}, \hdots , M_{k})$, which is equivalent to saying that $\mu \in FlDr(M)$  by \ref{thm:FlagDressian_equ}.
\end{proof}

We acknowledge that stipulating that a sequence of flag matroids form a valuated flag matroid is a very strong condition and it might be worthwhile to explore if a refinement of the statement above is possible with weaker conditions.

There has been recent work on extending the equivalent conditions for points residing in the Flag Dressian $FlDr(M)$ to the case of the positive Flag Dressian $FlDr^{+}(M)$, specifically in the case of the \emph{complete tropical flag variety} \cite[Theorem A]{boretsky2022polyhedral, 10.1093/imrn/rnac349} and we wish to explore the relation between the positive Flag Dressian and Lorentzian polynomials in this case in future work.

Another important aspect closely related to stable polynomials that we want to highlight is their connection with \emph{hyperbolic polynomials} and \emph{positively hyperbolic} varieties, studied in detail in \cite{borcea2010multivariate} and \cite{rincon2021positively} respectively. We first recall the definitions of a \emph{hyperbolic polynomial}\cite{borcea2010multivariate, amini2018non} and \emph{positively hyperbolic} variety \cite{rincon2021positively}.

\begin{definition}
A homogeneous polynomial $h(x) \in \mathbb{R}[x_{1} , \cdots , x_{n}]$ is hyperbolic with respect to a vector $e \in \mathbb{R}^{n}$  if $h(e) \neq 0$, and if for all $x \in \mathbb{R}^{n}$ the univariate polynomial $t \rightarrow h(te - x)$ has only real zeros.    
\end{definition}

If h is a hyperbolic polynomial of degree d, then we may write

\[  h(te - x) = h(e) \prod_{j=1}^{d} (t- \lambda_{j}(x)) \]

where 

\[ \lambda_{\text{max}}(x) = \lambda_{1}(x) \geq \lambda_{2}(x)  \geq \cdots 
  \geq \lambda_{d}(x) = \lambda_{\text{min}}(x) \]

are called the \emph{eigenvalues} of $x$ (with respect to $e$). The \emph{hyperbolicity cone} of $h$ with respect to $e$ is the set $\Lambda_{+}(h, e) = \{ x \in \mathbb{R}^{n} : \lambda_{\text{min}}(x) \geq 0 \}$.

We now relate the connection between stable and hyperbolic polynomials \cite[Lemma 2.7]{amini2018non},

\begin{lemma}
Let $P \in \mathbb{R}[x_{1} , \cdots , x_{n}]$ be a homogeneous polynomial. Then $P$ is stable if and only if $P$ is hyperbolic with $\mathbb{R}^{n}_{+} \subseteq \Lambda_{+}$.    
\end{lemma}

We wish to explore these connections between hyperbolic polynomials and stable polynomials in the context of the Dressian in future work.

\begin{definition}
Let $X \subset \mathbb{C}^{n}$ be a variety which is equidimensional of codimension $c \leq n-1$. X is called \emph{positively hyperbolic} if for every linear subspace $L$ in the positive Grassmannian
$Gr^{+}(c, n)$ and every $x \in X$, the imaginary part Im(x) does not belong to $L \setminus \{0\}$. A projective variety in $\mathbb{P}^{n-1}$ is positively hyperbolic if its affine cone in $\mathbb{C}^{n}$ is.    
\end{definition}

A straightforward conclusion from \cite[Proposition 2.2]{rincon2021positively} and \cite[Theorem 1.1]{purbhoo2018total} is that the variety defined by the homogenous multi-affine polynomial in Equation \ref{eq:ply_nonnegative}, which \emph{represents} $V$ in $Gr(k,n)$, is positively hyperbolic if $V$ corresponds to a positroid.

Another aspect of our discussion is the classification of various classes of matroids on whether they are Rayleigh, strongly Rayleigh, or if they satisfy the half-plane property. Before moving on, we also briefly discuss the notion of \emph{balanced matroids} \cite{feder1992balanced} and how they are related to our discussion.

\begin{definition}
Let $\mathcal{M}$ be a matroid on the ground set $E$ with basis $\mathcal{B}$. Let $\text{Pr}(e)$ denote the probability of an element $e$ being present in basis element $B$, which is chosen uniformly at random. The matroid $\mathcal{M}(E, \mathcal{B})$ is said to satisfy the  \emph{negatively correlated} property if 

\[  \text{Pr}(ef) \leq \text{Pr}(e)\text{Pr}(f) \]

for all pair of distinct $e,f \in E$.
\end{definition}

\begin{definition}
A matroid $\mathcal{M}$ is said to be \emph{balanced} if all its minors including itself satisfy the negative correlation property.    
\end{definition}

It is clear from the definition that Rayleigh matroids are balanced. In \cite{choe2006rayleigh} the authors also provide an explicit example of a transversal matroid of rank 4 \cite[Proposition 5.9]{choe2004homogeneous} (also updated in \cite{huh2021correlation}), which is not balanced and hence is also not Rayleigh. We recall the example here. Let $\mathcal{L}$ be a matroid on the ground set $E = \{1,2, \cdots, 10, e, f\}$ and $\mathcal{B}$ denote its bases, where the bases are the four transversals to the sets $\{1, 2, 3, 4, f\} , \{5, 6, 7, f\}, \{8, 9, 10, f\}$ and $\{1, 2, 3, 5, 6, 8, 9, e, f\}$ 

Equivalently the Rayleigh condition can be restated in terms of probability distribution in the following way \cite{huh2021correlation} for $B \in \mathcal{B}$ chosen uniformly at random,

\[ \text{Pr}(i \in B, j \in B) \>\> \text{Pr}( i \not \in B, j \not \in B) \leq \text{Pr}(i \in B,  j \not \in B)  \>\> \text{Pr}(j \in B, i \not \in B)  \]

But,

\[ \text{Pr}(i \in B, j \in B) \>\> \text{Pr}( i \not \in B, j \not \in B) = 0.04355\] 

and 

\[ \text{Pr}(i \in B,  j \not \in B) \>\> \text{Pr}(j \in B, i \not \in B) = 0.04298\]

which shows that $L$ is not Rayleigh. However, we want to point out the nuance here that this still does not provide any counterexample to Conjecture \ref{conj:lorent_conj}, as we see clearly that $L$ is still $\frac{8}{7}$- Rayleigh. Therefore, one of the problems that we want to work on is 

\begin{problem}
Verify Conjecture \ref{conj:lorent_conj} for the class of transversal matroids.  
\end{problem}

In the class of transversal matroids, it is already known that lattice path matroids are Rayleigh \cite{cohen2015lattice}. Hence, we initially aim to consider families of transversal matroids that are not necessarily lattice path matroids, and try to verify if Conjecture \ref{conj:lorent_conj} holds true for them.

Also as previously mentioned the authors in \cite{choe2004homogeneous} suggested that it might be the case that a large subclass of transversal matroids satisfy the half-plane property. However, since all matroids that satisfy the half-plane property are Rayleigh and now we know that the class of transversal matroids is not Rayleigh, therefore it also does not satisfy the half-plane property. But we wish to modify the problem in the following way,

\begin{problem}
Find the largest class of transversal matroids for which it satisfies the half-plane property.
\end{problem}

Pertaining to positroids, we wish to work on the following problem

\begin{problem}\label{prob:pos_strong_ray}
 Are positroids strongly Rayleigh?   
\end{problem}

With the complete classification of matroids on at most eight elements in \cite{kummer2022matroids}, one is tempted to search for a counterexample for Problem \ref{prob:pos_strong_ray} by looking for a positroid in the list of 22 matroids in \cite[Theorem 5.13]{kummer2022matroids} which do not satisfy the half-plane property, and hence are not strongly Rayleigh. We did verify that none of these 22 matroids are positroids and hence no counterexample comes from this list. However, we want to highlight that this does not guarantee that a counterexample of Problem \ref{prob:pos_strong_ray} does not exist in the form of a matroid on eight elements, because the property of being a positroid is not invariant under matroid isomorphisms, and the classification in \cite{kummer2022matroids} relies on a listing of matroids up to isomorphisms.

\bibliographystyle{siam}
\bibliography{biblio.bib}

\end{document}